\documentclass[11pt,letterpaper]{amsart}%
\usepackage{amsmath}
\usepackage{amsfonts}
\usepackage{amssymb}
\usepackage{graphicx}
\usepackage{amsmath}
\usepackage{txfonts}
\usepackage{graphicx}
\usepackage{amsthm}
\usepackage{amsfonts}
\usepackage{amssymb}
\usepackage[utf8]{inputenc}
\usepackage{xcolor}
\usepackage{ifthen}%
\providecommand{\U}[1]{\protect\rule{.1in}{.1in}}
\providecommand{\U}[1]{\protect\rule{.1in}{.1in}}
\newif\ifcomments
\commentsfalse
\newboolean{bluecomments}
\setboolean{bluecomments}{true}
\newcommand{\bluecomment}[1]{\ifthenelse{\boolean{bluecomments}}{\textcolor{blue}{#1}}{}}
\newcounter{mycounter}
\numberwithin{mycounter}{section}
\newtheorem{theorem}[mycounter]{Theorem}
\theoremstyle{plain}
\newtheorem{corollary}[mycounter]{Corollary}

\newtheorem{lemma}[mycounter]{Lemma}

\numberwithin{equation}{section}
\theoremstyle{definition}
\newtheorem{definition}[mycounter]{Definition}

\newcommand{\R}{\mathbb{R}}
\newcommand{\bnab}{\bar \nabla} 
\newcommand{\wedges}{\wedge \dots \wedge} 

\begin{document}
\title[Normal-flat calibrated submanifolds]{Calibrated submanifolds with flat normal bundles}
\date{\today}
\author{W. Jacob Ogden}
\address{Department of Mathematics\\
University of Washington, Seattle, WA 98105}
\email{wjogden@uw.edu}

\begin{abstract}
We show that submanifolds of Euclidean space which are calibrated by a constant-coefficient differential form and have flat normal bundles are planes. In fact, in a Riemannian manifold equipped with a parallel calibration, a calibrated submanifold subject to the condition that all shape operators commute is totally geodesic. 
\end{abstract}
\maketitle

\section{Introduction} 
In the theory of minimal submanifolds, there is a stark contrast between hypersurfaces (submanifolds of codimension 1) and submanifolds of high codimension. In codimension 1, objects such as variation fields and the mean curvature can be identified with scalars, while the vectorial nature of these objects cannot be avoided in high codimension. Several landmark results in the field, such as Simons' proof of the Bernstein conjecture in dimensions up to $6$ \cite{Simons} and the curvature estimates of Schoen-Simon-Yau \cite{SchoenSimonYau} depend crucially on the fact that hypersurfaces have rank-1 normal bundles.

The situation for high-codimension minimal submanifolds is more complicated. Lawson and Osserman \cite{LawsonOsserman} gave an example of a nonflat graphical minimal cone in dimension 4 and codimension 3. Their example is the cone over the graph of a scaled Hopf map $S^3 \to S^2$. Barbosa \cite{Barbosa}, and later Fischer-Colbrie \cite{Fischer-Colbrie}, proved rigidity for graphical minimal cones of dimension 3 in arbitrary codimension. While graphical minimal hypersurfaces over convex domains are volume-minimizing, Osserman \cite{Osserman} gave an example of an entire codimension-2 minimal graph which is not stable.

In recent decades, there has been success extending results for hypersurfaces to high codimension under the simplifying assumption that the minimal submanifold has a flat normal bundle. Motivation for this extra hypothesis can be found by considering Simons' equation for the second fundamental form of a minimal submanifold \cite{Simons}. In codimension 1, the equation simplifies to the incredibly useful form 
$$  \triangle A = - |A|^2 A,$$ while in high codimension, the equation is much more complicated and involves commutators of shape operators. In Euclidean space, submanifolds with flat normal bundles have shape operators which commute, so the commutator terms in Simons' equation vanish. This allows one to recover useful tools such as an improved Kato inequality. Using this improved Kato inequality, Xin \cite{Xin} proved that a complete $n$-dimensional, $n\leq 5$, minimal submanifold with flat normal bundle in $\R^{n+m}$ which is locally a graph over a fixed $n$-plane and satisfies growth hypotheses for the volume and volume element is a plane. Later, the dimension restriction $n \leq 5$ was removed by Smoczyk, Wang, and Xin \cite{SmoczykWangXin}. Curvature estimates of Schoen-Simon-Yau were also extended to the setting of normal-flat high-codimension minimal submanifolds by Smoczyk, Wang, and Xin \cite{SmoczykWangXin} and by M.-T. Wang \cite{Wang}, who also proved that minimal graphs with flat normal bundles are stable. Fu \cite{Fu, Fu2} and Seo \cite{Seo} proved Bernstein-type theorems for minimal submanifolds with flat normal bundles under curvature decay and superstability assumptions. 

Among high-codimension minimal submanifolds, the calibrated submanifolds introduced by Harvey and Lawson \cite{HarveyLawson79, HarveyLawson82} are of great interest. A calibration on a Riemannian manifold $M$ is a closed $n$-form $\omega$ with the property that for any vectors $X_1, \dots X_n \in T_p M$, 
\begin{equation} \label{calibration} \omega (X_1, \dots, X_n ) \leq | X_1 \wedges X_n |. \end{equation} 
A submanifold $\Sigma$ is said to be calibrated if any oriented frame for $T\Sigma$ saturates \eqref{calibration}. The fundamental fact about calibrated submanifolds is that they are homologically volume-minimizing. Federer \cite{Federer} observed that complex submanifolds of K\"ahler manifolds are calibrated by exterior powers of the K\"ahler form due to the classical Wirtinger inequality. Harvey and Lawson gave several additional examples of constant-coefficient calibrations on $\R^m$ which admit large families of calibrated submanifolds: the special Lagrangian, associative, coassociative, and Cayley calibrations. These calibrations also exist as parallel forms on manifolds of special holonomy (see \cite{Joyce}, for instance). 

There is much interest in determining the extent to which Liouville-Bernstein rigidity theorems extend to calibrated submanifolds. The Lawson-Osserman Hopf cone is in fact coassociative \cite{HarveyLawson82}, giving an immediate counterexample to a general Bernstein theorem for calibrated submanifolds. Under various additional hypotheses, rigidity theorems for special Lagrangian graphs hold (\cite{Yuan, TsuiWang, Yuan2, WarrenYuan, OgdenYuan}), but explicit entire solutions constructed by Warren \cite{Warren} and Li \cite{Li2} show that the general Bernstein theorem cannot hold for special Lagrangian graphs. It is not known whether nonflat graphical special Lagrangian cones exist. Bernstein-type results for other classes of calibrated submanifolds were obtained recently by Lien and Tsai \cite{LienTsai}. 

Li \cite{Li} studied special Lagrangian cones with flat normal bundles and proved that graphical special Lagrangian cones in $\R^{n} \times \R^n$ with flat normal bundles are planes for $n \leq 7$. Inspired by this work, we investigate submanifolds with commuting shape operators which are calibrated by parallel calibrations. Our main result is:

\begin{theorem} Assume  $(M^{n+m}, g ) $ is a Riemannian manifold, $\omega $ is a parallel calibration on $M$, and $\Sigma^n \subset M $ is calibrated by $\omega$. If for any two vectors $V, W \in N_p \Sigma$, the shape operators $A_V $ and $A_W $ commute, then $\Sigma$ is totally geodesic. \end{theorem}

In the case that $M$ is flat, the assumption that the shape operators of $\Sigma$ commute is equivalent to the normal curvature of $\Sigma$ vanishing. In particular we have the following for submanifolds of Euclidean space. 

\begin{corollary} 
If $\Sigma^n \subset \R^{n+m}$ is calibrated by a constant-coefficient calibration and the normal bundle $N\Sigma$ is flat, then $\Sigma $ is a plane. 
\end{corollary} 

The proof is based on a well-known formula for the Laplacian of the volume element. In the cases of complex submanifolds and special Lagrangian submanifolds, we also provide simple algebraic proofs. 

\section{Preliminaries}
Let $(M,g)$ be a Riemannian manifold and $\Sigma \subset M$ a submanifold. The restriction to $\Sigma$ of the tangent bundle $TM$ decomposes into the tangent and normal bundles of $\Sigma$:
$$TM|_\Sigma = T\Sigma \oplus N \Sigma.$$

The Levi-Civita connection $\bnab$ on $M$ induces connections on $T \Sigma$ and $N \Sigma$ defined by 
$$ \nabla_X Y = ( \bnab _X Y )^T , \quad \nabla _X V = (\bnab _X V )^N,$$ 
respectively, for $X \in T\Sigma$, $Y$ a local section of $T \Sigma$, and $V$ a local section of $N\Sigma$. Here, $(\cdot)^T$ and $(\cdot )^N $ are the orthogonal projections onto $T\Sigma$ and $N\Sigma$. The connections $\bnab$ and $\nabla$ are related by the second fundamental form, the section of $T^\ast \Sigma \otimes T^\ast \Sigma \otimes N\Sigma$ defined by 
$$ B(X,Y) = ( \bnab_X Y )^N.$$
For a tangent field $Y$ and normal field $V$, differentiating the equation $\langle Y, V \rangle =0$ gives 
$$ 0 = X \langle Y, V \rangle = \langle \bnab_X Y ,V \rangle + \langle Y , \bnab_X V \rangle = \langle B (X,Y) , V \rangle + \langle Y, \bnab _X V \rangle.$$
Thus, for a normal vector $V$, the shape operator $A_V$ is defined as a section of $T^\ast \Sigma \otimes T\Sigma$ implicitly by 
\begin{equation} \langle A_V(X), Y \rangle = - \langle B (X,Y), V \rangle, \label{AB} \end{equation} 
or equivalently, $A_V(X) = (\bnab _X  V )^T. $ 

Throughout this paper, $\{E_i\}$ will always denote an orthonormal frame for $T\Sigma$ and $\{V _\alpha \}$ will always denote an orthonormal frame for $N\Sigma$. With such frames chosen,
 we denote the components of the second fundamental form by 
$ h^\alpha_{ij} = \langle B ( E_i , E_j ) , V_\alpha \rangle.$

The mean curvature of $\Sigma$ is defined as the trace of $B$, 
$$ H = \sum_i B ( E_i, E_i) = \sum_{i, \alpha } h_{ii}^\alpha V_\alpha. $$

The normal connection on $\Sigma$ can be used to define a curvature operator on sections of $N \Sigma$: 

$$R^N_{X,Y} V = \nabla_X \nabla_Y V - \nabla _Y \nabla_X V - \nabla_{[X,Y]} V .$$

\begin{definition} A submanifold $\Sigma$ has a flat normal bundle if for any $X,Y \in T_p \Sigma$, $V \in N_p \Sigma$, 
$$R^N_{X,Y} V = 0.$$
\end{definition} 

Let $\bar R$ denote the Riemann curvature of $M$. The normal curvature of $\Sigma$ is related to the curvature of $M$ by the Ricci equation for $\Sigma$ \cite{Spivak}:
\begin{lemma}\label{NormalCurvature} If $X,Y \in T_p \Sigma$ and $V,W \in N_p \Sigma$, then 
$$ \langle   R^N _{XY }V , W  \rangle = \langle \bar R_{XY} V, W \rangle + \langle X , [ A_W, A_V ] (Y) \rangle .$$

\end{lemma} 

\begin{proof} Extend $V$ to a local section of $N \Sigma$, and extend $X$ and $Y$ to local sections of $T\Sigma$ with $[X,Y]=0$. Using the definitions of the normal connection and shape operator, 
\begin{equation*} 
\begin{aligned} 
 \nabla_X \nabla_Y V - \nabla _Y \nabla _X V  
 & = \nabla_X ( \bnab _Y V - ( \bnab_Y V )^T ) - \nabla_Y ( \bnab _X V - ( \bnab_X V )^T ) \\
 & = \nabla_X ( \bnab _Y V - A_V (Y)  ) - \nabla_Y ( \bnab _X V - A_V (X)) \\
 & = [ \bnab_X ( \bnab _Y V - A_V (Y)  )]^N -[ \bnab_Y ( \bnab _X V - A_V (X))]^N \\
 & = (\bar R_{XY} V )^N - B ( X, A_V (Y) ) + B ( Y , A_V (X)).
\end{aligned} 
\end{equation*} 
Testing against $W$, 
\begin{equation*} 
\begin{aligned} 
 \langle R^N_{XY} V , W \rangle &=  \langle \bar R_{XY} V , W \rangle -  \langle  B ( X , A_V (Y ) ) , W \rangle + \langle  B ( Y , A_V (X ) ) , W \rangle \\ 
 & = \langle \bar R_{XY} V , W \rangle + \langle  X, A_W(A_V(Y)) \rangle-  \langle  A_W ( Y), A_V ( X)  \rangle 
 \end{aligned} 
\end{equation*} 
 using \eqref{AB}. 
 Finally, use the fact that $A_V$ is self-adjoint to obtain 
 $$  \langle R^N_{XY} V , W \rangle =  \langle \bar R_{XY} V , W \rangle + \langle  X, [A_W, A_V](Y) \rangle. \qedhere$$
 \end{proof} 
 
 From Lemma \ref{NormalCurvature}, we see that if $M$ is flat ($\bar R =0$), then $\Sigma$ has flat normal bundle if and only if $[A_W , A_V ]=0$ for any $V,W \in N_p \Sigma.$ 

The useful fact for our purposes is that commuting endomorphisms can be simultaneously diagonalized, so submanifolds whose shape operators all commute admit local frames in which all of the shape operators are diagonal. 

We conclude this section by computing the Laplacian of the tangent plane of $\Sigma$ as a section of the bundle $\bigwedge^n ( TM | _\Sigma) .$ See \cite{Xin, Wang} for a similar calculation or \cite{Fischer-Colbrie} for a dual version.

\begin{lemma}\label{Laplace} If $\Sigma \subset M$ is minimal, then 
$$ \triangle ( E_1 \wedges E_n )  = - |B|^2  E_1 \wedges E_n  + \sum _i \sum_{ j \neq k , \alpha , \beta } h^\alpha _{ ij } h^\beta _{ik}  E_1 \wedges V_\alpha \wedges V_\beta \wedges E_n, $$
where $V_\alpha$ and $V_\beta$ appear in the $j^\text{th}$ and $k^\text{th}$ positions in the wedge product, respectively. 
\end{lemma} 

\begin{proof} We calculate at a fixed point $p$ with the orthonormal frames $\{E_i\}$ and $\{V_\alpha\}$ chosen so that $\nabla E_i= 0 $ and $\nabla V_\alpha=0$ at $p$. First note that in an orthonormal frame, 
\begin{equation} \label{Christoffel} 2 \langle E_j , \nabla _{E_i} E_j \rangle = E_i \langle E_j , E_j \rangle =0.\end{equation} 
Differentiating the tangent plane once, 
\begin{equation*} 
\begin{aligned} 
 \bnab_{E_i } ( E_1 \wedges E_n ) &= \sum_j  \big [ E_1 \wedges (\bnab _{E_i } E_j)^T \wedges E_n \\ 
 & \quad  +  E_1 \wedges (\bnab _{E_i } E_j)^N \wedges E_n \big] .\\
\end{aligned}
\end{equation*} 
For the first term, only the $E_j$ component of $(\bnab _{E_i } E_j)^T$ produces a nonzero wedge product, but by \eqref{Christoffel}, the $E_j$ component vanishes. For the terms involving the normal part, we substitute $(\bnab _{E_i } E_j)^N = h^\alpha _{ij} V_\alpha$ (here we use the Einstein summation convention for Greek indices). Now we take another covariant derivative: 

\begin{align}
\bnab _{E_i } (E_1 \wedges (h_{ij}^\alpha V_\alpha) \wedges E_n) 
& = E_1 \wedges \bnab _{E_i }( h_{ij}^\alpha V_\alpha ) \wedges E_n \notag \\& \quad +  \sum_{k \neq j } E_1 \wedges h_{ij} ^ \alpha V_\alpha  \wedges h_{ik} ^ \beta V_\beta \wedges E_n . \label{SecondDerivative}
\end{align} 

We have 
\begin{align*}  \bnab_{E_i} (  h_{ij} ^\alpha V_\alpha ) &= E_i \langle \bnab _{E_i } E_j , V_\alpha \rangle V_\alpha + h_{ij}^\alpha \bnab_{E_i} V_\alpha \\
& = E_i \langle \bnab_{E_j } E_i , V_\alpha \rangle V_\alpha - \sum_k h_{ij}^\alpha h_{ik}^\alpha E_k \\
& = - E_i \langle E_i , \bnab_{E_j} V_\alpha \rangle V_\alpha -  \sum_k h_{ij}^\alpha h_{ik}^\alpha E_k\\
& = - \langle \bnab_{E_i} E_i , \bnab_{E_j } V_\alpha \rangle - \langle E_i , \bnab_{ E_i } \bnab_{E_j } V_{\alpha}\rangle  V_\alpha  - \sum_k h_{ij}^\alpha h_{ik}^\alpha E_k \\ 
& = - \langle E_i , \bnab_{E_j} \bnab_{E_i} V_\alpha \rangle V_\alpha - \sum_k h_{ij}^\alpha h_{ik}^\alpha E_k \\ 
& = - E_j \langle E_i , \bnab_{E_i } V_\alpha \rangle V_\alpha +\langle  \bnab_{E_j} E_i , \bnab_{E_i } V_\alpha \rangle V_\alpha - \sum_k h_{ij} ^\alpha  h_{ik}^\alpha E_k 
\\ & = - E_j ( h_{ii}^\alpha ) - \sum_k h_{ij} ^\alpha  h_{ik}^\alpha E_k .
\end{align*} 
Substituting this expression into the wedge product in \eqref{SecondDerivative}, only the $E_j$ term of the sum survives, 
so 
\begin{align} 
\bnab _{E_i } (E_1 \wedges (h_{ij}^\alpha V_\alpha) \wedges E_n) 
& = - \sum_\alpha ( h_{ii}^\alpha+ (h_{ij}^\alpha) ^ 2 ) E_1 \wedges E_n \notag \\& \quad +  \sum_{k \neq j } E_1 \wedges h_{ij} ^ \alpha V_\alpha  \wedges h_{ik} ^ \beta V_\beta \wedges E_n . \label{SecondDerivative2} 
\end{align} 
Upon summing \eqref{SecondDerivative2} over $i$, we have 
$$ \sum_i h_{ii}^\alpha = \langle H , V_\alpha \rangle =0 $$ since $\Sigma$ is minimal.
Now since $|B|^2 = \sum_{i,j,\alpha} (h_{ij}^\alpha)^2 $, summing over $j$ proves the lemma. 
\end{proof} 

\section{Proof of theorem 1.1} 

\begin{proof} Since $\omega$ calibrates $\Sigma$, 
$$ \omega ( E_1\wedge \dots\wedge E_n ) = 1.$$
Combining this with the assumption that $\omega$ is parallel, 
\begin{equation*} \begin{aligned} 
0 & = E_i [\omega ( E_1 \wedge\dots\wedge E_n ) ] \\
&= \bnab_{E_i } ( \omega ) ( E_1 \wedges E_n ) + \omega ( \bnab_{E_i} ( E_1 \wedges E_n ))\\ 
& = \omega ( \bnab_{E_i} ( E_1 \wedges E_n )). \end{aligned} \end{equation*} 
Differentiating again and summing,  
$$ \triangle [ \omega ( E_1 \wedge \dots \wedge E_n ) ] = \omega ( \triangle ( E_1 \wedge \dots \wedge E_n) ) .$$
By Lemma \ref{Laplace}, 
$$  \triangle ( E_1 \wedge \dots \wedge E_n) = \sum_{i} \sum_{j \neq k , \alpha , \beta } h_{ij}^\alpha h_{ik}^\beta ( E_1 \wedge \dots \wedge V_\alpha \wedge \dots \wedge V_\beta \wedge \dots \wedge E_n ) - |B|^2 ( E_1 \wedge \dots \wedge E_n) .$$
Since $\Sigma$ has commuting shape operators, the frame $\{E_i\}$ can be chosen so that $[h^\alpha_{i j} ]$ is diagonal for all $\alpha.$ But since $j \neq k$ in the summation, one of $h^\alpha _{ij} $ and $h^\beta _{ik}$ must be zero in each term. Therefore 
$$ 0=\triangle [\omega ( E_1 \wedge \dots \wedge E_n )] = - |B|^2,$$
so $\Sigma$ is totally geodesic.
\end{proof} 

For the cases of complex submanifolds or minimal Lagrangian submanifolds of a K\"ahler manifold, this rigidity result can be observed directly from algebraic consequences of the interaction between the normal bundle and the ambient complex structure without taking the Laplacian of the tangent plane. 

Let $J$ denote the ambient complex structure. If $M$ is K\"ahler and $\Sigma$ is a complex submanifold, then $J$ preserves the tangent and normal bundles of $\Sigma.$ If the shape operators of $\Sigma$ commute, then for any choice of normal vector $V$, the operators $A_V$ and $A_{JV}$ can be simultaneously diagonalized,
but 
$$ \langle \bnab _{E_i} V, E_i \rangle = \langle \bnab _{E_i} JV, JE_i \rangle = \langle \bnab_{JE_i} JV, E_i \rangle,$$ using the fact that $J$ is parallel, so 
$ A_{JV}$ is diagonal only if $A_V =0$.

If $\Sigma$ is instead Lagrangian, then $J$ is an isomorphism between $T\Sigma$ and $N\Sigma$, so we may translate any calculation involving normal vectors to an intrinsic calculation. In particular, 
$ J ( B( X,Y ))$ takes values in $T\Sigma$. 
We can then define a section $\tilde B$ of $\bigotimes^3 T^\ast \Sigma$ by 
$$ \tilde B ( X,Y,Z) = \langle B ( X,Y ) , J Z \rangle.$$
Since $$\langle B ( X,Y ) , J Z \rangle = \langle \bnab _X Y , J Z \rangle =- \langle Y , J \bnab_X Z  \rangle = \langle J Y , B ( X,Z) \rangle ,$$
the tensor $\tilde B$ is symmetric in all three arguments. This means that $h_{ij}^\alpha = h_{i \alpha }^j .$ But if $[ h_{ij}^\alpha]$ is diagonal for each $\alpha$, then the only nonzero components are $h_{ii}^i$, which can be seen by choosing $i \neq \alpha$ in the previous identity. By minimality, $$0=\sum _{ i} h_{ii}^\alpha= h_{\alpha \alpha}^\alpha,$$
so the entire second fundamental form vanishes. 

McLean \cite[Prop. 4.2]{McLean} showed that the normal bundle of a coassociative submanifold $\Sigma$ of a $G_2$ manifold is isomorphic to the bundle of antiselfdual 2-forms, and hence is also an intrinsic object. It is therefore natural to ask whether there is a simple algebraic proof that coassociative submanifolds with commuting shape operators are totally geodesic. 

\textbf{Acknowledgment.} This work is supported by the NSF Graduate Research Fellowship Program under grant DGE-2140004 and by NSF grant DMS-2453862. The author thanks Yu Yuan for helpful discussions.
\bibliographystyle{amsplain}
\bibliography{nfc}

\end{document}